\documentclass[preprint,12p,3p]{elsarticle}
\usepackage{graphicx}
\usepackage{amssymb}
\usepackage{amsthm}
\usepackage{amsmath}
\usepackage{lineno}
\usepackage{hyperref}
\usepackage{verbatim}

\newtheorem{theorem}{Theorem}[section]

\theoremstyle{definition}

\theoremstyle{remark}
\newtheorem{remark}[theorem]{Remark}
\newtheorem{corollary}[theorem]{Corollary}
\newtheorem{proposition}[theorem]{Proposition}
\numberwithin{equation}{section}

\usepackage{graphicx, pdflscape}
\usepackage{amsmath, hyperref}
 \usepackage{amssymb} 
\usepackage{relsize, float, lineno}

 \usepackage{xcolor} 
\usepackage{subcaption}
\allowdisplaybreaks

\begin{document}
\begin{frontmatter}

\title{A Modified SIS Epidemic Model with Application to\\ Health Insurance Pricing}

\author[VIASM2]{Tuan Chau Do}
\ead{tuanb2610@gmail.com} 
\address[VIASM2]{Faculty of Natural Sciences, Budapest University of Technology and Economics
1111 Budapest, Muegyetem rkp. 3, Hungary}
\author[VIASM1]{Tien Thinh Le}
\ead{letienthinh2510@gmail.com} 
\address[VIASM1]{Faculty of Mathematics and Statistics, University of Economics Ho Chi Minh City,\\ 279 Nguyen Tri Phuong Campus, Ward 8, District 10, Ho Chi Minh City, Vietnam}
\author[HUS]{Nguyen Trong Hieu}
\ead{hieunguyentrong@gmail.com}
\address[HUS]{University of Science, Vietnam National University - Hanoi, 334 Nguyen Trai Street,\\ Thanh Xuan Trung Ward, Thanh Xuan District, Hanoi, Vietnam}
\author[FPT]{Manh Tuan Hoang\corref{Corr}}
\cortext[Corr]{Corresponding author}
\ead{tuanhm16@fe.edu.vn} 
\address[FPT]{Department of Mathematics, FPT University, Hoa Lac Hi-Tech Park, Km29 Thang Long Blvd, Hanoi, Viet Nam}


\begin{abstract}
In this work, we investigate a modified version of the classical SIS model that incorporates hospitalization for treatment and disease-induced mortality, aiming to more accurately capture the dynamics relevant to health insurance pricing models. More precisely, we introduce a new framework, referred to as the SISHD model, which considers both hospitalized individuals and disease-induced mortality. Dynamical properties of the proposed model are thoroughly analyzed, including positivity and boundedness of the solutions, the basic reproduction number, the existence and asymptotic stability of equilibrium points. Furthermore, we utilize the proposed model in the context of health insurance pricing, where the population sizes estimated from the SISHD model are used to determine appropriate insurance costs. Finally, numerical simulations are conducted to illustrate and validate the theoretical results.
\end{abstract}





\end{frontmatter}



\section{Introduction}\label{Sec1}
Mathematical epidemiology has a long-standing and well-established history of development and continues to be an important field with numerous practical applications (see, e.g., \cite{Brauer1, Brauer2, Martcheva} and references therein). The well-known SIR model, first proposed by Kermack and McKendrick \cite{Kermack}, represents one of the earliest frameworks in epidemic modeling, and it is widely used to introduce basic principles of epidemiological modeling. In practice, mathematical modeling and analysis of infectious diseases have become essential for understanding transmission dynamics, forecasting outbreaks, and informing strategies for disease control and prevention. In recent years, applications of mathematical epidemiology models in insurance pricing have attracted considerable attention from many researchers; as a result, several findings have proven to be highly useful not only in theory but also in practice (see, for instance, \cite{Chernov, Feng, Feng1, Francis, Nkeki, Nkeki1, Zhai, Zinihi} and references therein).

In this work, we revisit a classical SIS epidemic model \cite{Allen, Brauer1, Brauer2, Martcheva}, which describes infectious diseases in which individuals recover but do not acquire immunity after recovery. The model under consideration of the form:
\begin{equation}\label{eq:1}
\begin{split}
&S^{\prime}(t) = \Lambda - \beta I(t)S(t) + \alpha_I I(t) - \mu S(t),\\
&I^{\prime}(t) = \beta I(t)S(t) - (\alpha_I + \gamma_I + \mu)I(t)
\end{split}
\end{equation}
subject to initial data: $S(0) \geq 0$ and $I(0) \geq 0$.

In this model:
\begin{itemize}
\item the entire population is divided into two disjoint groups: the \textit{susceptible class (S)}, consisting of individuals who are not yet infected but are at risk of contracting the disease, and the \textit{infected class (I)}, consisting of individuals who are infected and capable of transmitting the disease to others;
\item $S(t)$ and $I(t)$ stand for the number of susceptible individuals and infected individuals at time $t$, respectively;
\item $\beta$ is the transmission rate;
\item $\alpha_I$ is the recovery rate;
\item $\mu$ is the natural death rate;
\item $\gamma_I$ is the disease-induced death in the infectious class $(I)$;
\item $\Lambda$ denotes the recruitment rate (births).
\end{itemize}
It is important to note that all the parameters are assumed to be positive due to epidemiological reasons.

The transmission diagram of \eqref{eq:1} is given in Figure \ref{Fig:Diagram1}. More details of the derivation and dynamical behaviour of the model \eqref{eq:1} can be found in \cite{Allen, Brauer1, Brauer2, Martcheva} (see also \cite{Gray, Lan} and references therein). SIS models, including \eqref{eq:1}, are recognized as fundamental frameworks that form the basis of many important epidemiological models and have a wide range of practical applications.
\begin{figure}[H]
\centering
\includegraphics[height=5.0cm,width=12cm]{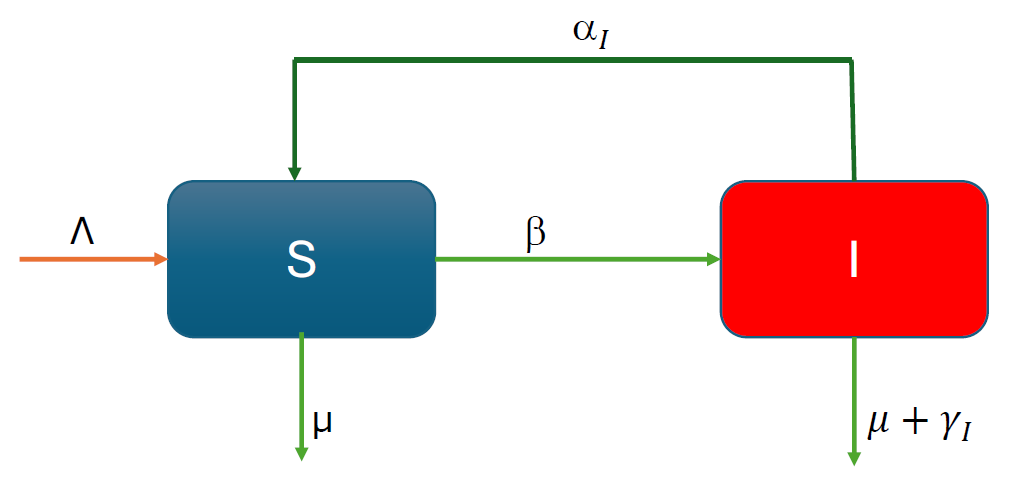}
\caption{Transmission diagram for the model \eqref{eq:1}.}\label{Fig:Diagram1}
\end{figure}

In this study, we consider a modified version of \eqref{eq:1} and study its applications in actuarial disease insurance. For this purpose, we first modify \eqref{eq:1} to make it more suitable for reality. To this end, we incorporate the following new classes into the model:
\begin{itemize}
\item the class $H$ consists of individuals who are hospitalized for treatment and may still transmit the disease to others;
\item the class $D$ consists of individuals who have died from the disease.
\end{itemize}

Then, an extended version of \eqref{eq:1} is given by:
\begin{equation}\label{eq:2}
\begin{split}
&S' = \Lambda - \beta IS - \beta \epsilon HS + \alpha_I I - \mu S + \alpha_H H,\\
&I' = \beta I S + \beta \epsilon HS - (\alpha_I + \gamma_I + \mu + \delta)I,\\
&H' = \delta I - (\gamma_H + \mu + \alpha_H)H,\\
&D' = \gamma_I I + \gamma_H H,
\end{split}
\end{equation}
where
\begin{itemize}
\item $\delta$ is the hospitalization rate from $I$ to $H$;
\item $\epsilon \in [0, 1]$ is the relative infectiousness of hospitalized individuals ($H$) compared with infectious individuals ($I$)
\item $\alpha_H$ is the discharge (recovery) rate from $H$ to $S$;
\item $\gamma_H$ is the disease-induced death rate for hospitalized individuals $H$.
\end{itemize}
The transmission between the classes $S, I, H$ and $D$ in \eqref{eq:2} is described in Figure \ref{Fig:Diagram2}. 
\begin{figure}[H]
\centering
\includegraphics[height=5.5cm,width=12cm]{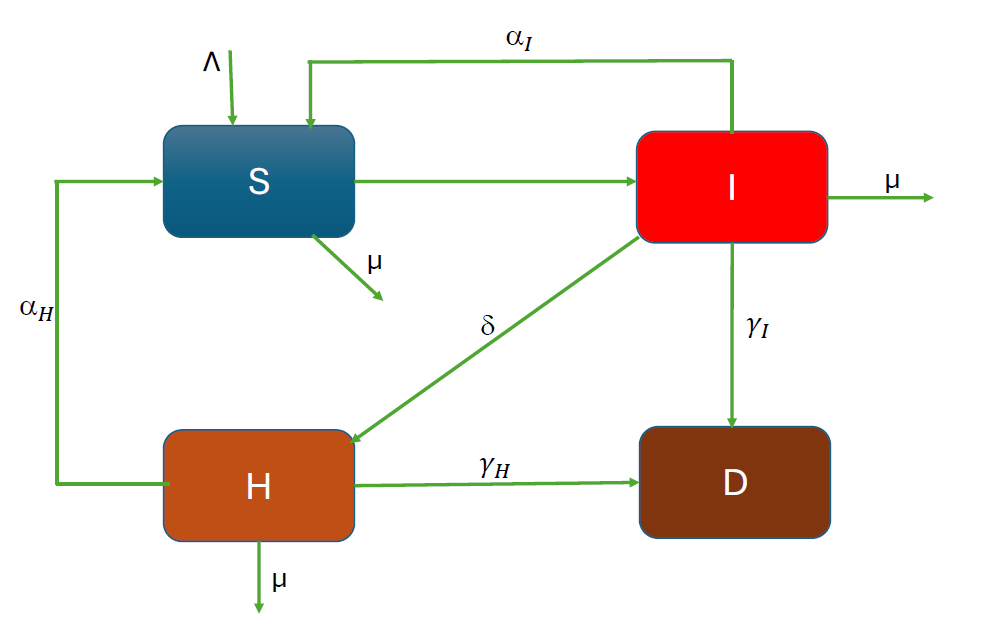}
\caption{Transmission diagram for the model \eqref{eq:2}.}\label{Fig:Diagram2}
\end{figure}
After performing a qualitative analysis of model \eqref{eq:2}, we consider its application to health insurance pricing, in which the number of individuals estimated from \eqref{eq:2} is used to calculate insurance premiums. The results obtained suggest several potential and practical applications in real-world settings.

The outline of this article is as follows:\\
In Section \ref{Sec2}, we establish dynamical properties of the model \eqref{eq:2}, including positivity and boundedness of solutions, the basic reproduction number, equilibrium points and their asymptotic stability. Section \ref{Sec3} provides an application of \eqref{eq:2} in actuarial disease insurance. Section \ref{Sec4} reports  numerical simulations to support and illustrative theoretical findings. The last sections provides some including remarks and open problems.
\section{Dynamical analysis of SISHD model}\label{Sec2}
In this section, we analyze dynamical properties of the model \eqref{eq:2}. We denote by $N(t)$ the total living individuals in \eqref{eq:2}, that is $N(t) = S(t) + I(t) + H(t)$ for $t \geq 0$.
\begin{proposition}\label{Proposition1}
The model \eqref{eq:2} admits the set $\mathbb{R}_+^4 = \big\{(S, I, H, D)^T \in \mathbb{R}^4|S, I, H, D \geq 0\big\}$ as a positively invariant set, that is $\big(S(t),\,I(t),\,H(t),\,D(t)\big)^T \in \mathbb{R}_+^4$ for $t > 0$ whenever $\big(S(0),\,I(0),\,H(0),\,D(0)\big)^T \in \mathbb{R}_+^4$. Also, we have the following estimate for $N(t)$
\begin{equation}\label{eq:3}
\limsup_{t \to \infty}N(t) \leq \dfrac{\Lambda}{\mu}.
\end{equation}
\end{proposition}
\begin{proof}
Let $\big(S(0),\,I(0),\,H(0),\,D(0)\big)^T$ be any initial data belonging to $\mathbb{R}_+^4$. It follows from \eqref{eq:2} that
\begin{equation*}
\begin{split}
&S'|_{S = 0} = \Lambda + \alpha_I I + \alpha_H H \geq 0\\
&I'|_{I = 0} = \beta \epsilon HS \geq 0\\
&H'|_{H = 0} = \delta I \geq 0,\\
&D'|_{D = 0} = \gamma_I I + \gamma_H H \geq 0
\end{split}
\end{equation*}
for all $S, I, H, D \geq 0$. Thus, we deduce from \cite[Proposition B.7]{Smith} that $S(t), I(t), H(t), D(t) \geq 0$ for $t \geq 0$. This means that $\mathbb{R}_+^4 $ is a positively invariant set of \eqref{eq:2}.

It is easy to verify that the total living individuals satisfies
\begin{equation*}
N' = (S + I + H)' = \Lambda - \mu(S + I + H) - \gamma_I I - \gamma_H H \leq \Lambda - \mu(S + I + H) = \Lambda - \mu N.
\end{equation*}
Using the basic comparison theorem for ODEs \cite{McNabb} yields
\begin{equation*}
N(t) \leq \bigg(N(0) - \dfrac{\Lambda}{\mu}\bigg)e^{-\mu t} + \dfrac{\Lambda}{\mu},
\end{equation*}
which implies that the estimate \eqref{eq:3}.
Thus, the proof is complete.
\end{proof}

Since the $D$ component only appears in the fourth equation, it is sufficient to consider the first three equations:
\begin{equation}\label{eq:1new}
\begin{split}
&S' = \Lambda - \beta IS - \beta \epsilon HS + \alpha_I I - \mu S + \alpha_H H,\\
&I' = \beta I S + \beta \epsilon HS - (\alpha_I + \gamma_I + \mu + \delta)I,\\
&H' = \delta I - (\gamma_H + \mu + \alpha_H)H
\end{split}
\end{equation}
on its feasible region defined by
\begin{equation}\label{eq:5}
\Omega = \bigg\{(S, I, H) \in \mathbb{R}_+^3|S + I + H \leq \dfrac{\Lambda}{\mu}\bigg\}.
\end{equation}
We now determine the set of equilibrium points and the basic reproduction number of \eqref{eq:1new}.
\begin{theorem}\label{Theorem1}
(i) The model \eqref{eq:1new} always possesses a disease-free equilibrium (DFE) point $E_0 = (S_0,\, I_0,\, H_0) = \big(\frac{\Lambda}{\mu},\,0,\,0\big)$ for all values of the parameters.\\
(ii) The basic reproduction number of \eqref{eq:1new} is computed by
\begin{equation}\label{eq:6}
\mathcal{R}_0 = \dfrac{
\beta \Lambda \left[ (\gamma_H + \mu + \alpha_H) + \epsilon \delta \right]
}{
\mu (\alpha_I + \gamma_I + \mu + \delta)(\gamma_H + \mu + \alpha_H)
}.
\end{equation}
(ii) A unique disease-endemic equilibrium (DEE) point $E_* = (S_*, I_*, H_*)$ exists if and only if $\mathcal{R}_0 > 1$. Moreover, if $E_*$ exists, it is given by
\begin{equation}\label{eq:7}
\begin{split}
S^* &= \dfrac{(\alpha_I + \gamma_I + \mu + \delta)(\gamma_H + \mu + \alpha_H)}
{\beta\,[\,\gamma_H + \mu + \alpha_H + \epsilon\,\delta\,]}, \\
I^* &=  \dfrac{(\gamma_H + \mu + \alpha_H)\,
\big(\Lambda - \mu\,S^*\big)}
{(\gamma_I + \mu + \delta)(\gamma_H + \mu + \alpha_H) - \alpha_H\,\delta},\\
H^* &= \dfrac{\delta\,I^*}{\gamma_H + \mu + \alpha_H}.
\end{split}
\end{equation}
\end{theorem}
\begin{proof}
\textbf{Proof of Part (i).} Any equilibrium point of \eqref{eq:1new} is a solution to the system
\begin{equation}\label{eq:8}
\begin{split}
&0 = \Lambda - \beta IS - \beta \epsilon HS + \alpha_I I - \mu S + \alpha_H H,\\
&0 = \beta I S + \beta \epsilon HS - (\alpha_I + \gamma_I + \mu + \delta)I,\\
&0 = \delta I - (\gamma_H + \mu + \alpha_H)H.
\end{split}
\end{equation}
It follows from the 3rd equation of \eqref{eq:8} that
\begin{equation}\label{eq:9a}
H = \dfrac{\delta}{\gamma_H + \mu + \alpha_H}I.
\end{equation}
Substituting this equation into the second equation of \eqref{eq:8} yields
\begin{equation}\label{eq:9}
I\bigg[\beta S + \beta\epsilon\dfrac{\delta}{\gamma_H + \mu + \alpha_H}S  - (\alpha_I + \gamma_I + \mu + \delta)\bigg] = 0.
\end{equation}
It is clear that \eqref{eq:9} always has a trivial solution $I = 0$. From this, we obtain $H = 0$. Substituting $I = H =0$ into the first equation \eqref{eq:8} yields $S = \Lambda/\mu$. Therefore, $(S_0, I_0, H_0) = (\Lambda/\mu,\,0,\,0)$ forms a DFE point of \eqref{eq:1new}.\\
\textbf{Proof of Part (i).} We now use a method developed in \cite{vandenDriessche} to compute the basic reproduction number of \eqref{eq:1new}. 
By reordering the variables in \eqref{eq:1new} as $X = (I, H, S)^T$, the DFE point is transformed to $X_f = (0, 0, \Lambda/\mu)$ and \eqref{eq:1new} can be written in the matrix form 
\begin{equation*}
X' = \mathcal{F}(X) - \mathcal{V}(X),
\end{equation*}
where
\begin{equation*}
\mathcal{F}(X) = 
\begin{pmatrix}
\beta IS + \beta\epsilon HS\\
\\
0\\
\\
0
\end{pmatrix},\quad
\mathcal{V}(X) =
\begin{pmatrix}
(\alpha_I + \gamma_I + \mu + \delta)I\\
\\
-\delta I + (\gamma_H + \mu + \alpha_H)H\\
\\
-\Lambda + \beta IS + \beta \epsilon HS - \alpha_I I + \mu S - \alpha_H H
\end{pmatrix}.
\end{equation*}
Consequently,
\begin{equation*}
D\mathcal{F}(X_f) = 
\begin{pmatrix}
\beta\dfrac{\Lambda}{\mu}& \beta\epsilon\dfrac{\Lambda}{\mu}&0\\
&&\\
0&0&0\\
&&\\
0&0&0
\end{pmatrix},
\quad
D\mathcal{V}(X_f) = 
\begin{pmatrix}
\alpha_I + \gamma_I + \mu + \delta&0&0\\
&&\\
-\delta&\gamma_H + \mu + \alpha_H&0\\
&&\\
\beta\dfrac{\Lambda}{\mu} - \alpha_I&\beta\epsilon\dfrac{\Lambda}{\mu} - \alpha_H&\mu
\end{pmatrix}.
\end{equation*}
Hence, the basic reproduction number is computed by
\begin{equation*}
\mathcal{R}_0 = \rho(FV^{-1}) =
\dfrac{
\beta \Lambda \left[ (\gamma_H + \mu + \alpha_H) + \epsilon \delta \right]
}{
\mu (\alpha_I + \gamma_I + \mu + \delta)(\gamma_H + \mu + \alpha_H)
}.
\end{equation*}
(iii) From \eqref{eq:9}, we obtain a non-trivial solution
\begin{equation*}
S_* = \dfrac{(\alpha_I + \gamma_I + \mu + \delta)(\gamma_H + \mu + \alpha_H)}{\beta \big[(\gamma_H + \mu + \alpha_H) + \epsilon \delta \big]}.
\end{equation*}
Substituting this equation and \eqref{eq:9a} into the first equation of \eqref{eq:8} yields
\begin{equation*}
I^* = \dfrac{\Lambda - \mu S^*}{\gamma_I + \mu + \delta - \dfrac{\alpha_H \delta}{\gamma_H + \mu + \alpha_H}}
\end{equation*}
It is easily seen that $I_* > 0$ if and only if $\mu S_* < \Lambda$, which is equivalent to $\mathcal{R}_0 > 1$. Hence, we conclude that $E_* = (S_*, I_*, H_*)$ forms a unique DEE point of \eqref{eq:1new} if and only if $\mathcal{R}_0 > 1$.

The proof is complete.
\end{proof}
Applying Theorem 1 in \cite{vandenDriessche} to \eqref{eq:1new} yields:
\begin{corollary}\label{Corollary1}
The DFE point $E_0$ is locally asymptotically stable with respect to the set $\Omega$ if $\mathcal{R}_0 < 1$ and is unstable if $\mathcal{R}_0 > 1$.
\end{corollary}
We now analyze global asymptotic stability of the DFE point.
\begin{theorem}
The DFE point $E_0$ is not only locally asymptotically stable but also globally asymptotically stable if $\mathcal{R}_0 < 1$.
\end{theorem}
\begin{proof}
To show the GAS of the DFE point, we will verify that \eqref{eq:1new} satisfies Conditions $(H1)-(H2)$ given in \cite{Castillo-Chavez}. Indeed, we rewrite \eqref{eq:1new} in the form
\begin{equation}\label{eq:10}
\begin{split}
&X' = F(X, Z),\\
&Z' = G(X, Z),
\end{split}
\end{equation}
where
\begin{equation*}
\begin{split}
&X = S, \quad Z = (I, H)^T,\\
&F(X, Z) = \Lambda - \beta IS - \beta \epsilon HS + \alpha_I I - \mu S + \alpha_H H,\\
&G(X, Z) =
\begin{pmatrix}
\beta I S + \beta \epsilon HS - (\alpha_I + \gamma_I + \mu + \delta)I\\
\delta I - (\gamma_H + \mu + \alpha_H)H
\end{pmatrix}.
\end{split}
\end{equation*}
Then, the DFE point is transformed to $U_0 = (X^*, 0) =  (\Lambda/\mu,\,\,0)$.

Note that the subsystem $X' = F(X, 0)$ is given by
\begin{equation*}
X' = \Lambda - \mu X.
\end{equation*}
Consequently, $X^*$ is a globally asymptotically stable equilibrium point of this equation. Hence, Condition $(H1)$ is satisfied.

To verify Condition $(H2)$, we rewrite the function $G(X, Z)$ in the form
\begin{equation*}
G(X, Z) = AZ - \widehat{G}(X, Z),
\end{equation*}
where
\begin{equation*}
A =
\begin{pmatrix}
\beta X^* - (\alpha_I + \gamma_I + \mu + \delta)&\beta\epsilon X^*\\
&\\
\delta&-(\gamma_H + \mu + \alpha_H)
\end{pmatrix},\quad
\widehat{G}(X, Z) =
\begin{pmatrix}
\beta I(X^* - X) + \beta\epsilon H(X^* - X)\\
&\\
0
\end{pmatrix}.
\end{equation*}
It is easy to verify that $A$ is an $M$-matrix and $\widehat{G}(X, Z) \geq 0$ for all $(X, Z)$ in $\Omega$, which implies that Condition $(H2)$ holds. 

Hence, by using the global stability results developed in \cite{Castillo-Chavez}, we obtain the desired conclusion. The proof is complete.
\end{proof}
We now consider the case $\mathcal{R}_0 > 1$ and investigate asymptotic stability of the DEE point $E_*$.
\begin{theorem}\label{Theorem2}
The DEE point $E_*$ is locally asymptotically stable if and only if $\mathcal{R}_0 > 1$.
\end{theorem}
\begin{proof}
The Jacobian matrix of \eqref{eq:1new} is given by
\[
J(S, I, H) =
\begin{pmatrix}
-\beta I - \beta \epsilon H - \mu &
-\beta S + \alpha_I &
-\beta \epsilon S + \alpha_H \\
\beta I + \beta \epsilon H &
\beta S - (\alpha_I + \gamma_I + \mu + \delta) &
\beta \epsilon S \\
0 & \delta & -(\gamma_H + \mu + \alpha_H)
\end{pmatrix},
\]
which implies that
\[
J(S_*, I_*, H_*) =
\begin{pmatrix}
-\beta I_* - \beta \epsilon H_* - \mu &
-\beta S_* + \alpha_I &
-\beta \epsilon S_* + \alpha_H \\
\beta I_* + \beta \epsilon H_* &
\beta S_* - (\alpha_I + \gamma_I + \mu + \delta) &
\beta \epsilon S_* \\
0 & \delta & -(\gamma_H + \mu + \alpha_H)
\end{pmatrix}.
\]
By setting
\begin{equation*}
\begin{split}
&M_0 = \Lambda - \mu S_*,\\
&M_1 = \dfrac{\gamma_H + \mu + \alpha_H}{(\gamma_I + \mu + \delta)(\gamma_H + \mu + \alpha_H) - \alpha_H\delta},\\
&M_2 = \dfrac{\delta}{\gamma_H + \mu + \alpha_H},\\
&B = M_2,
\end{split}
\end{equation*}
we obtain $I_* = M_1(\Lambda - \mu S_*) = M_1M_0$ and $H_* = M_1M_2(\Lambda - \mu S_*) = M_0M_1M_2$. Using this relation reduces $J(S_*, I_*, H_*)$ to
\[
J(S_*, M_0) =
\begin{pmatrix}
-\beta\,M_0\,M_1\,(1+\epsilon B)-\mu 
& -\beta S_*+\alpha_I 
& -\beta\epsilon S_*+\alpha_H \\[6pt]
\beta\,M_0\,M_1\,(1+\epsilon B) 
& \beta S_*-(\alpha_I+\gamma_I+\mu+\delta) 
& \beta\epsilon S_* \\[6pt]
0 & \delta & -(\gamma_H+\mu+\alpha_H)
\end{pmatrix},
\]

The characteristic polynomial of $J(S_*,I_*,H_*)$ is given by $p(\lambda) = \lambda^3 + a_1\lambda^2 + a_2\lambda + a_3$, where
\begin{equation*}
\begin{split}
a_1 &= (\alpha_H+\alpha_I+\delta+\gamma_H+\gamma_I+3\mu)\;-\;\beta S_* \;+\;\beta M_0 M_1\;+\;\beta\epsilon B\,M_0 M_1,\\
a_2 &= \alpha_H\alpha_I+\alpha_H\delta+\alpha_H\gamma_I+\alpha_I\gamma_H+\delta\gamma_H+\gamma_H\gamma_I + 2\mu(\alpha_H+\alpha_I+\delta+\gamma_H+\gamma_I)+3\mu^2 \\
&+(\beta S_*)\big(-\alpha_H-\gamma_H-2\mu-\epsilon\delta\big) + (\beta M_0 M_1)\big(\alpha_H+\delta+\gamma_H+\gamma_I+2\mu\big)\\
&+(\beta\epsilon B\,M_0 M_1)\big(\alpha_H+\delta+\gamma_H+\gamma_I+2\mu\big),\\
a_3 &= \mu^3+\mu^2(\alpha_H+\alpha_I+\delta+\gamma_H+\gamma_I) +\mu(\alpha_H\alpha_I+\alpha_H\delta+\alpha_H\gamma_I+\alpha_I\gamma_H+\delta\gamma_H+\gamma_H\gamma_I) \\
&+(\beta S_*)\big(-\mu^2-\mu(\alpha_H+\gamma_H+\epsilon\delta)\big)  +(\beta M_0 M_1)\big(\mu^2+\mu(\alpha_H+\delta+\gamma_H+\gamma_I)+\alpha_H\gamma_I+\delta\gamma_H+\gamma_H\gamma_I\big) \\
& +(\beta\epsilon B\,M_0 M_1)\big(\mu^2+\mu(\alpha_H+\delta+\gamma_H+\gamma_I)+\alpha_H\gamma_I+\delta\gamma_H+\gamma_H\gamma_I\big).
\end{split}
\end{equation*}
We now show that $a_1 > 0, a_3 > 0$ and $a_1a_2 - a_3 > 0$. Indeed, the second equation of \eqref{eq:8} implies that
\[
\beta S_* \bigl(1+\epsilon B \bigr) \;=\; \alpha_I + \gamma_I + \mu + \delta,
\]
or equivalent to
\[
\beta S_* \;=\; \frac{\alpha_I + \gamma_I + \mu + \delta}{1+\epsilon B}.
\]
Substituting this into $a_1$ yields
\[
a_1 
= \beta M_0 M_1 \,(1+\epsilon B) 
+ \mu 
+ (\gamma_H+\mu+\alpha_H) 
+ \frac{\epsilon B}{1+\epsilon B}\,(\alpha_I+\gamma_I+\mu+\delta) > 0.
\]

It is easy to see that
\[
(\beta S_*)\bigl[\mu^2+\mu(\alpha_H+\gamma_H+\epsilon\delta)\bigr]
= \mu\,(\alpha_I+\gamma_I+\mu+\delta)\,(\gamma_H+\mu+\alpha_H),
\]
which implies that 
\[
a_3 \;=\; \beta\,M_0\,M_1\,
\left(1+\frac{\epsilon\,\delta}{\gamma_H+\mu+\alpha_H}\right)\,
\Bigl[\mu^2+\mu(\alpha_H+\delta+\gamma_H+\gamma_I)+\alpha_H\gamma_I+\delta\gamma_H+\gamma_H\gamma_I\Bigr] > 0.
\]

By some simple algebraic manipulations, we can show that $a_1a_2 - a_3 > 0$.

Thus, we have shown that $a_1$, $a_3 > 0$ and $a_1a_2 - a_3 > 0$. We deduce from Routh-Hurwitz criteria (Theorem 4.4 in \cite{Allen}) that all of the roots of the characteristic polynomial are negative or have negative real part. This means that $E_*$ is locally asymptotically stable. The proof is complete.
\end{proof}
\begin{remark}\label{Remark1}
It is not easy to establish the GAS of the DFE equilibrium point; however, numerical simulations reported in Section \ref{Sec3} suggest that this equilibrium is not only locally asymptotically stable but also globally asymptotically stable. This is consistent with well-known results on the asymptotic stability of epidemiological models \cite{Allen, Brauer1, Brauer2, Martcheva}.

Since the DFE point is globally asymptotically stable, in practical situations we should try to control the parameters such that the basic reproduction number is less than $1$. In this case, the DFE point is globally asymptotically stable, implying that the disease will be eradicated. Through the sensitivity analysis of $\mathcal{R}_0$ with respect to the controllable parameters, we can obtain useful insights for disease prevention. More clearly,  the derivative of $\mathcal{R}_0$ with respect to the controllable parameters $(\beta,\, \epsilon,\, \alpha_I,\, \alpha_H,\, \gamma_I,\, \gamma_H)$ are given by
\[
\begin{aligned}
\frac{\partial \mathcal{R}_0}{\partial \beta}
&= \frac{\Lambda \,\big[(\gamma_H + \mu + \alpha_H) + \epsilon \delta\big]}
{\mu \, (\alpha_I + \gamma_I + \mu + \delta)(\gamma_H + \mu + \alpha_H)} \;>\; 0, \\[8pt]
\frac{\partial \mathcal{R}_0}{\partial \epsilon}
&= \frac{\beta \Lambda \delta}
{\mu \, (\alpha_I + \gamma_I + \mu + \delta)(\gamma_H + \mu + \alpha_H)} \;>\; 0, \\[8pt]
\frac{\partial \mathcal{R}_0}{\partial \alpha_I}
&= -\,\frac{\beta \Lambda \,\big[(\gamma_H + \mu + \alpha_H) + \epsilon \delta\big]}
{\mu \, (\alpha_I + \gamma_I + \mu + \delta)^{2}(\gamma_H + \mu + \alpha_H)} \;<\; 0, \\[8pt]
\frac{\partial \mathcal{R}_0}{\partial \gamma_I}
&= -\,\frac{\beta \Lambda \,\big[(\gamma_H + \mu + \alpha_H) + \epsilon \delta\big]}
{\mu \, (\alpha_I + \gamma_I + \mu + \delta)^{2}(\gamma_H + \mu + \alpha_H)} \;<\; 0, \\[8pt]
\frac{\partial \mathcal{R}_0}{\partial \alpha_H}
&= -\,\frac{\beta \Lambda \,\epsilon \delta}
{\mu \, (\alpha_I + \gamma_I + \mu + \delta)(\gamma_H + \mu + \alpha_H)^{2}} \;<\; 0, \\[8pt]
\frac{\partial \mathcal{R}_0}{\partial \gamma_H}
&= -\,\frac{\beta \Lambda \,\epsilon \delta}
{\mu \, (\alpha_I + \gamma_I + \mu + \delta)(\gamma_H + \mu + \alpha_H)^{2}} \;<\; 0.
\end{aligned}
\]
From this, we can see the influence of each parameter on the disease transmission process. In numerical simulations performed in Section \ref{Sec4}, we will compute sensitivity indices using a particular set of the parameters and show the influence of the controllable parameters.
\end{remark}
\section{A Health Insurance Pricing Application}\label{Sec3}
Suppose a policyholder enrolls in an annual insurance plan with a fixed premium. If the policyholder contracts an infectious disease and requires medical treatment, the insurer will provide coverage for the treatment costs, subject to the fulfillment of all policy conditions  \cite{Chernov, Feng, Feng1, Francis, Nkeki, Nkeki1, Zhai, Zinihi}.
%
Assume an insurance model in which premiums are collected from the susceptible population group (class $S$) at a constant rate of $\pi$ per unit of time. In exchange, the insurer provides benefit payments as follows:

\begin{itemize}
    \item a benefit amount $b_I$ for infected individuals ($I$);
    \item a benefit amount $b_H$ for hospitalized patients ($H$);
    \item in the event of death caused by the disease, a lump-sum payment of $d$ will be made.
\end{itemize}
According to the fundamental principle of equivalence \cite{Feng}, the level premium is determined such that the expected present value of benefit payments equals the expected present value of premium income
\[
\mathbb{E}[\text {Present value of benefit outgo}]=\mathbb{E}[\text{Present value of premium income}].
\]
Hence, the {{equivalence principle}} equation for this insurance model can be expressed as follows:
\begin{equation}\label{eq:Pi}
\pi \int_0^T S(t) dt = b_I \int_0^T I(t) dt + b_H \int_0^T H(t) dt + d \int_0^T D(t) dt.
\end{equation}
From this equation, we obtain a premium rate that represents the zero-profit condition for insurance companies over a coverage period of length $T$:
\begin{equation}\label{eq:Pi1}
\pi=\dfrac{b_I \int_0^T I(t) dt + b_H \int_0^T H(t) dt + d \int_0^T [\gamma_I I(t) + \gamma_H H(t)]dt}{\int_0^T S(t) dt}.
\end{equation}
After determining the zero-profit premium level and the associated loadings, the reserve level becomes the primary consideration for insurance companies in practice. The reserve level is established as
\begin{equation}\label{Pi2}
V(t) = \pi{\int_t^T S(\tau) d\tau} - \bigg[b_I \int_t^T I(\tau) d\tau + b_H \int_t^T H(\tau) d\tau + d \int_t^T (\gamma_I I(\tau) + \gamma_H H(\tau))d\tau\bigg].
\end{equation}
The reserve level serves as a key indicator reflecting the insurer’s financial equilibrium throughout the coverage period \cite{Zhai}:
\begin{itemize}
\item $V(t) < 0$: The insurer maintains a negative reserve at time $t$ indicating that the expected liabilities exceed the premium inflows. Such a condition necessitates additional capital to preserve solvency and is indicative of a prudent or conservative pricing approach.
\item $V(t) > 0$: The insurer holds a positive reserve, implying that the present value of premium income exceeds the expected future liabilities. Although this may appear to indicate a surplus, it often reflects an overpriced premium structure, which raises concerns regarding fairness, regulatory compliance, and long-term financial sustainability. In actuarial practice, persistently negative reserves are generally deemed unacceptable.
\item $V(t) = 0$: This represents the ideal actuarial equilibrium, in which the premium is precisely calibrated to align with the expected future liabilities. It adheres to the principle of equivalence, which posits that, in a fairly priced insurance model, the present value of benefits and expenses must equal the present value of premium income.
\end{itemize}
Hence, if $\pi$ is set too low, the reserve $V$ may turn negative during the simulation period, indicating that the scheme is underfunded and unable to meet its expected liabilities. Conversely, if $\pi$ is set excessively high, the reserve may surpass the expected claims by a considerable margin, suggesting overpricing or operational inefficiency.

To determine an admissible value of $\pi$ for a $T$-day insurance policy, 
the premium rate should be chosen such that $V(t) \geq 0$ for all $t \in [0, T]$. 
This leads to the following actuarial constraint:
\begin{equation}\label{eq:Pi3}
\pi \geq \pi^* :=  \max_{t \in[0, T]}\dfrac{b_I \int_t^T I(\tau) d\tau + b_H \int_t^T H(\tau) d\tau + d\int_t^T [\gamma_I I(\tau) + \gamma_H H(\tau)]d\tau}{\int_t^T S(\tau) d\tau},
\end{equation}
which guarantees that the discounted value of future premiums remains sufficient to cover future liabilities. It further defines a practical upper limit for $\pi$ which can be estimated numerically from the model trajectories.
\section{Numerical Experiments}\label{Sec4}
In this section, we conduct numerical examples to support and illustrate the theoretical assertions.
\subsection{Dynamics of the SISHD model when $\mathcal{R}_0 < 1$}
In this subsection, we investigate dynamical behaviour of the proposed SISHD model \eqref{eq:2} when the basic reproduction number $\mathcal{R}_0$ less than $1$. For this purpose, we consider \eqref{eq:2} with some parameter sets given in Table \ref{Table1}. For these parameter sets, $\mathcal{R}_0 < 1$; consequently, the DFE point is globally asymptotically stable. Also, the sets of initial data given in Table \ref{Table2} are used in numerical simulations.
\begin{table}[H]
\centering
\caption{Five parameter sets for which $\mathcal{R}_0 < 1$}
\begin{tabular}{ccccccccccccccccccccccccccccccccccccc}
\hline
\textbf{Set} & $\Lambda$ & $\mu$ & $\beta$ & $\epsilon$ & $\alpha_I$ & $\gamma_I$ & $\delta$ & $\gamma_H$ & $\alpha_H$ & $\mathcal{R}_0$ \\
\hline
A1 & 20.0 & 0.02 & 0.00012 & 0.20 & 0.05 & 0.10 & 0.02 & 0.05 & 0.01 & 0.6632 \\
A2 & 20.0 & 0.02 & 0.00018 & 0.25 & 0.06 & 0.12 & 0.03 & 0.06 & 0.02 & 0.8413 \\
A3 & 30.0 & 0.03 & 0.00010 & 0.15 & 0.04 & 0.08 & 0.01 & 0.04 & 0.01 & 0.6367 \\
A4 & 10.0 & 0.01 & 0.00020 & 0.30 & 0.08 & 0.15 & 0.02 & 0.05 & 0.02 & 0.8269 \\
A5 & 25.0 & 0.025 & 0.00014 & 0.10 & 0.05 & 0.09 & 0.03 & 0.06 & 0.015 & 0.7395 \\
\hline
\end{tabular}
\label{Table1}
\end{table}
\begin{table}[H]
\centering
\caption{Initial data $(S_0, I_0, H_0)$ used for numerical simulations of the SIH model \eqref{eq:1new}.}
\label{Table2}
\begin{tabular}{ccccccccccccccccccccccccccc}
\hline
\textbf{Initial set} & $S_0$ & $I_0$ & $H_0$ \\
\hline
IC$_1$ & 800 & 100 & 100 \\
IC$_2$ & 700 & 200 & 50  \\
IC$_3$ & 500 & 250 & 250 \\
IC$_4$ & 600 & 100 & 300 \\
IC$_5$ & 400 & 300 & 300 \\
\hline
\end{tabular}
\end{table}
In all numerical examples reported below, we use the classical fourth-order Runge-Kutta method (\cite{Ascher}) with a small step size, namely $h = 10^{03}$,  to numerically solve \eqref{eq:1new} over the interval $[0, 365]$. The obtained solutions of \eqref{eq:1new} are depicted in Figures \ref{Fig:1}-\ref{Fig:5}.
It is clear that the solutions are positive, stable and convergent to the DEE point. This is evidence supporting the theoretical findings presented in Section \ref{Sec2}.
\begin{figure}[H]
\centering
\includegraphics[height=9.5cm,width=12cm]{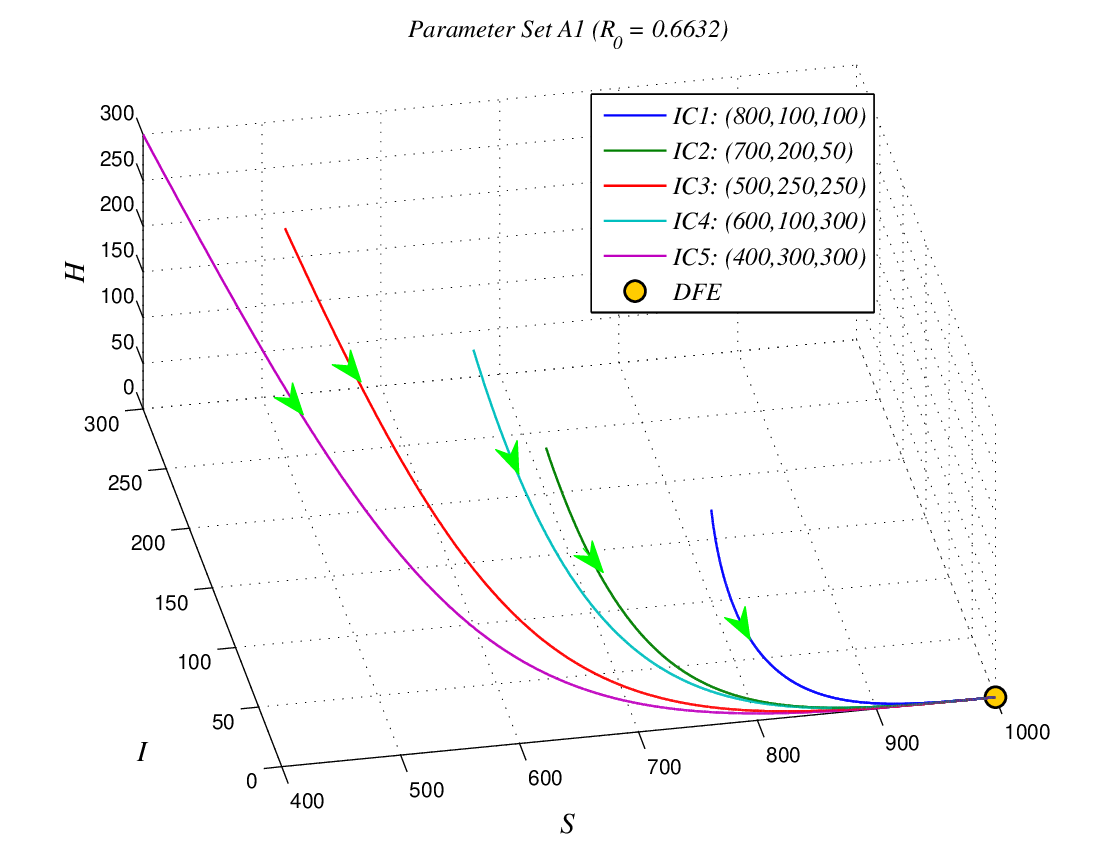}
\caption{Dynamics of \eqref{eq:1new} for the parameter set A1 in Table \ref{Table1}}\label{Fig:1}
\end{figure}
\begin{figure}[H]
\centering
\includegraphics[height=9.5cm,width=12cm]{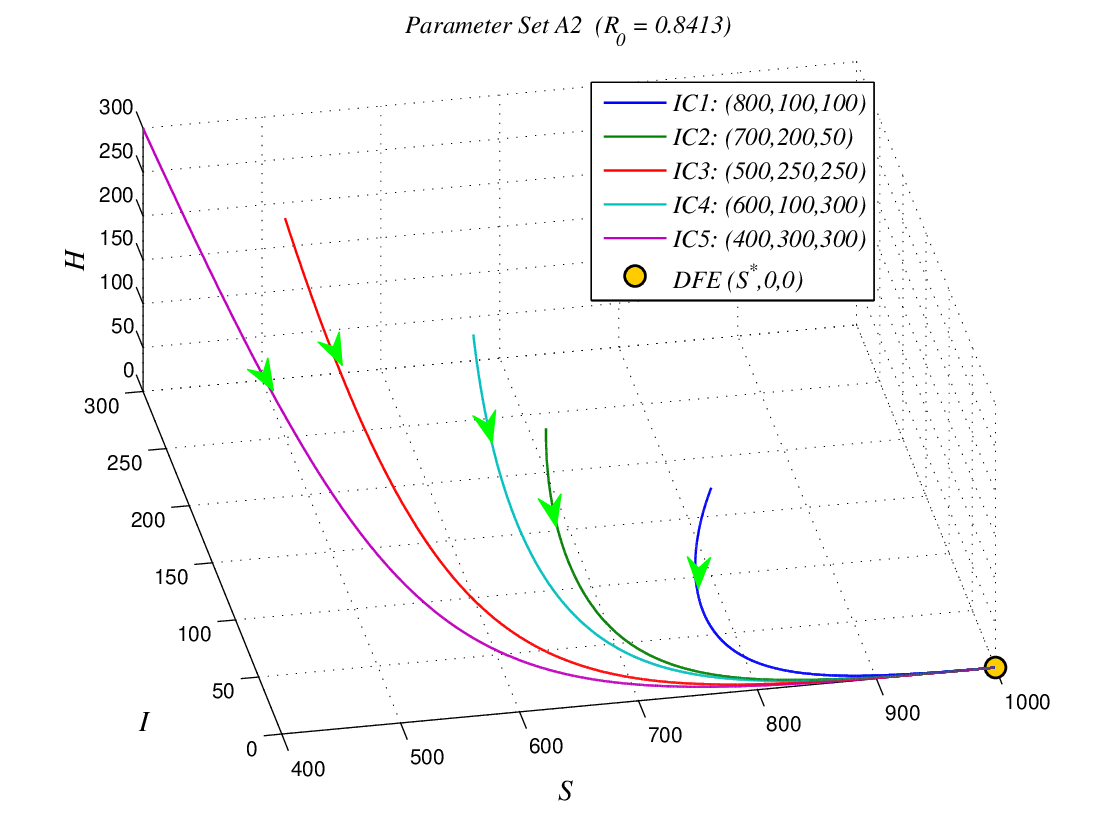}
\caption{Dynamics of \eqref{eq:1new} for the parameter set A2 in Table \ref{Table1}}\label{Fig:2}
\end{figure}
\begin{figure}[H]
\centering
\includegraphics[height=9.5cm,width=12cm]{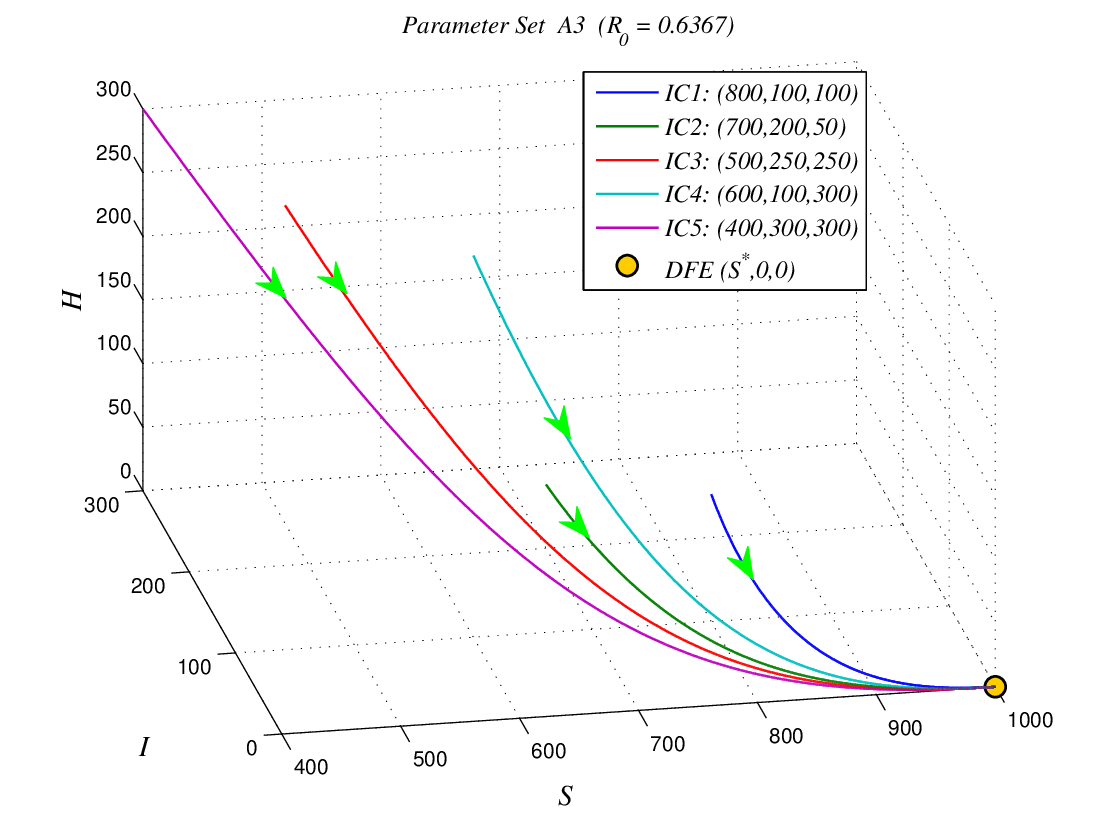}
\caption{Dynamics of \eqref{eq:1new} for the parameter set A3 in Table \ref{Table1}}\label{Fig:3}
\end{figure}
\begin{figure}[H]
\centering
\includegraphics[height=9.5cm,width=12cm]{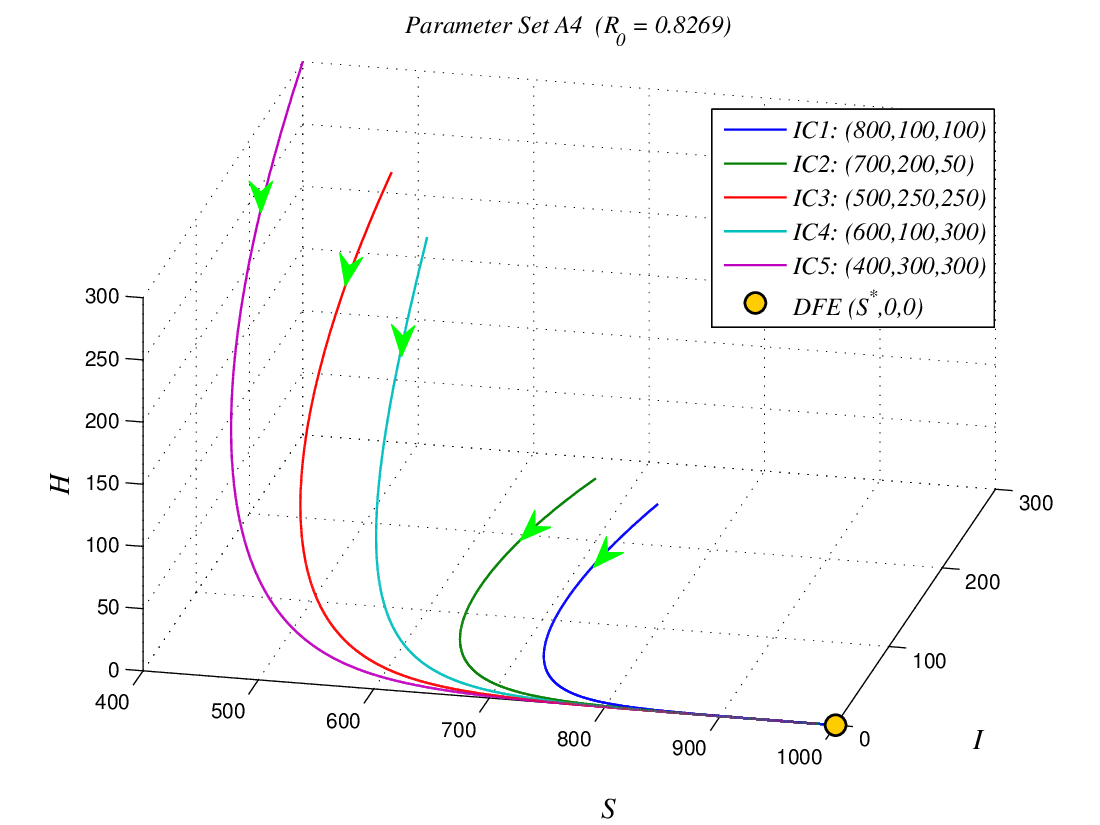}
\caption{Dynamics of \eqref{eq:1new} for the parameter set A4 in Table \ref{Table1}}\label{Fig:4}
\end{figure}
\begin{figure}[H]
\centering
\includegraphics[height=9.5cm,width=12cm]{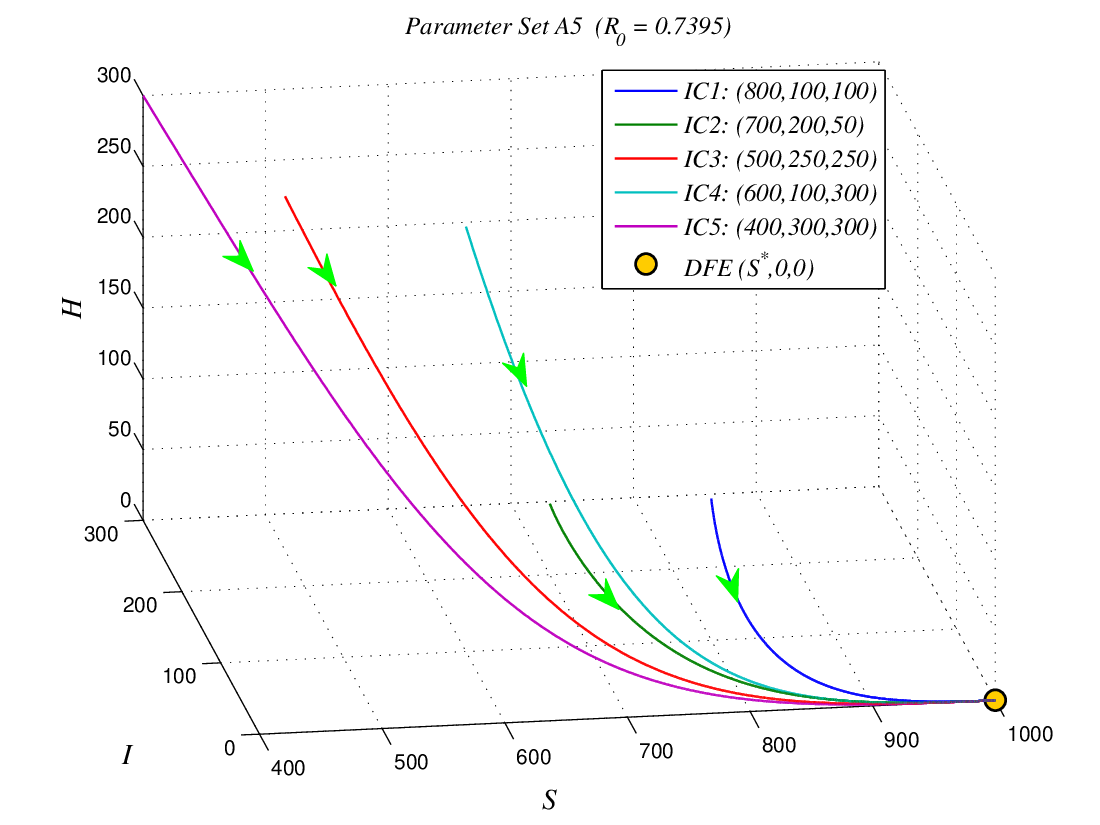}
\caption{Dynamics of \eqref{eq:1new} for the parameter set A5 in Table \ref{Table1}}\label{Fig:5}
\end{figure}
\subsection{Dynamics of the SISHD model when $\mathcal{R}_0 > 1$}
In this subsection, we investigate dynamical behaviour of the SISHD model \eqref{eq:2} under the condition $\mathcal{R}_0 > 1$. The model is evaluated using parameter sets and initial data given in Table \ref{Table3} and \ref{Table4}, respectively. The solutions generated by employing the RK4 method are shown in Figures \ref{Fig:6}-\ref{Fig:10}. It is important to remark that the numerical results suggest that the DEE point is not only locally asymptotically stable but also globally asymptotically stable. This supports the conjecture stated in Remark \ref{Remark1}.

\begin{table}[H]
\centering
\caption{Five parameter sets for which $\mathcal{R}_0 > 1$ and corresponding disease-endemic equilibria $(S^*, I^*, H^*)$}
\label{Table3}
\begin{tabular}{lccccccccccccc}
\hline
\textbf{Set} & $\Lambda$ & $\mu$ & $\beta$ & $\epsilon$ & $\alpha_I$ & $\gamma_I$ & $\delta$ & $\gamma_H$ & $\alpha_H$ & $\mathcal{R}_0$ & $S^*$ & $I^*$ & $H^*$ \\
\hline
$B_1$ & 20 & 0.02 & 0.000154 & 0.20 & 0.02 & 0.05 & 0.02 & 0.03 & 0.005 & 1.5 & 666.67 & 75.60 & 27.49 \\
$B_2$ & 20 & 0.02 & 0.000253 & 0.25 & 0.03 & 0.06 & 0.03 & 0.04 & 0.010 & 2.0 & 500.00 & 94.59 & 40.54 \\
$B_3$ & 30 & 0.03 & 0.000239 & 0.15 & 0.01 & 0.04 & 0.02 & 0.03 & 0.005 & 2.5 & 400.00 & 203.48 & 62.61 \\
$B_4$ & 10 & 0.01 & 0.000318 & 0.30 & 0.02 & 0.05 & 0.02 & 0.04 & 0.010 & 3.5 & 285.71 & 93.17 & 31.06 \\
$B_5$ & 25 & 0.025 & 0.000649 & 0.10 & 0.03 & 0.05 & 0.03 & 0.04 & 0.010 & 5.0 & 200.00 & 198.02 & 79.21 \\
\hline
\end{tabular}
\end{table}
\begin{table}[H]
\centering
\caption{Initial conditions $(S_0, I_0, H_0)$ used in numerical simulations}\label{Table4}
\begin{tabular}{cccccccccccccc}
\hline
\textbf{Initial data} & $S_0$ & $I_0$ & $H_0$ &\\
\hline
IC1 & 800 & 100 & 100 & \\
IC2 & 700 & 200 & 50  & \\
IC3 & 500 & 250 & 250 &\\
IC4 & 600 & 100 & 300 &\\
IC5 & 400 & 300 & 300 &\\
\hline
\end{tabular}
\end{table}
\begin{figure}[H]
\centering
\includegraphics[height=9.5cm,width=12cm]{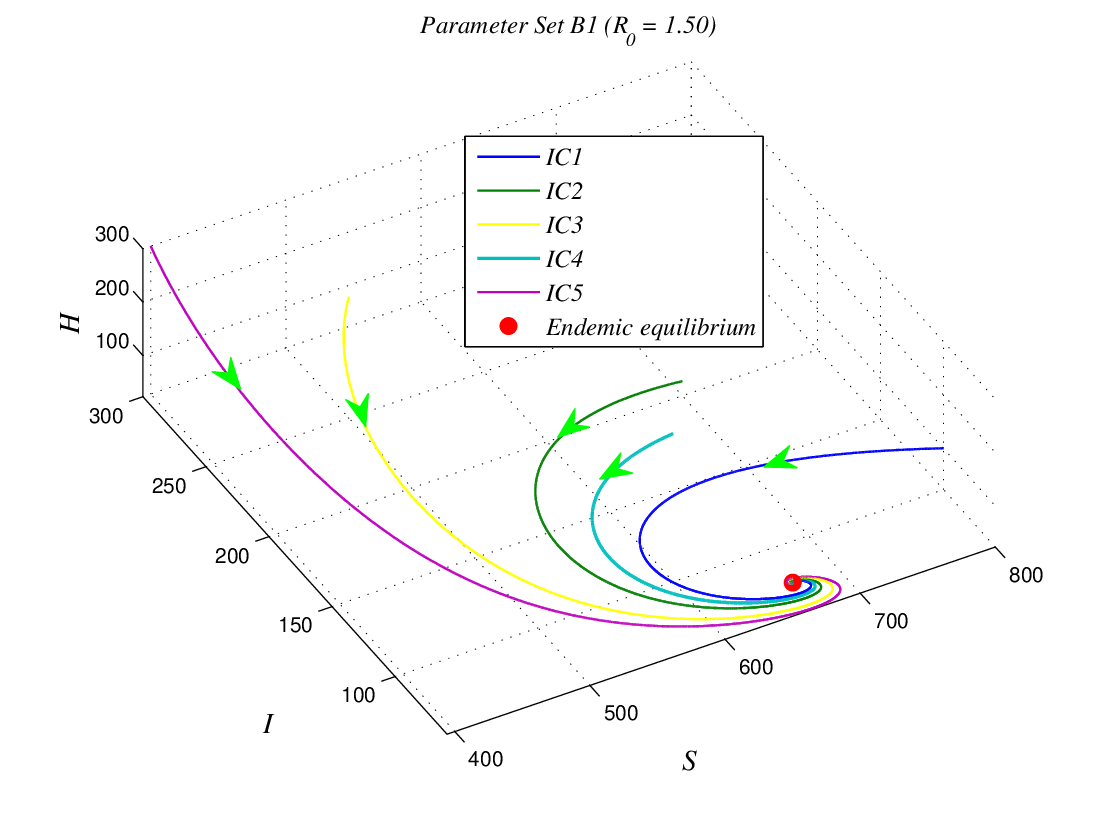}
\caption{Dynamics of \eqref{eq:1new} for the parameter set B1 in Table \ref{Table3}}\label{Fig:6}
\end{figure}
\begin{figure}[H]
\centering
\includegraphics[height=9.5cm,width=12cm]{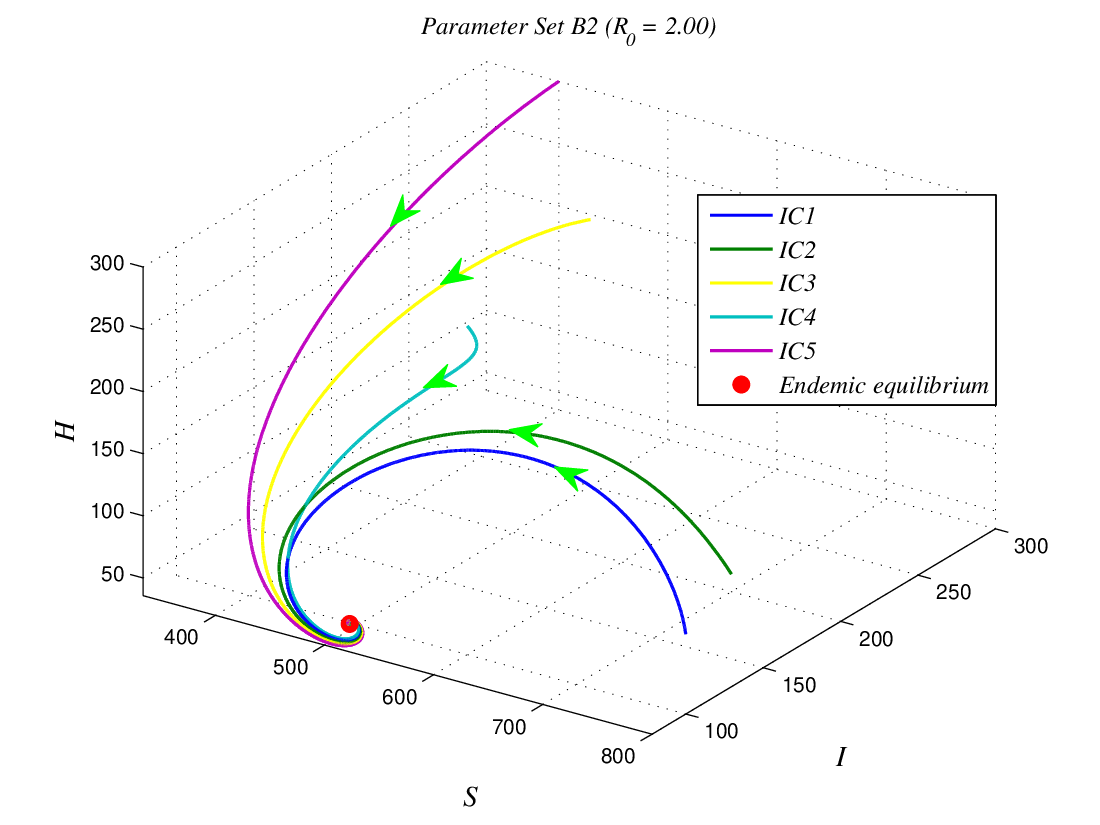}
\caption{Dynamics of \eqref{eq:1new} for the parameter set B2 in Table \ref{Table3}.}\label{Fig:7}
\end{figure}
\begin{figure}[H]
\centering
\includegraphics[height=9.5cm,width=12cm]{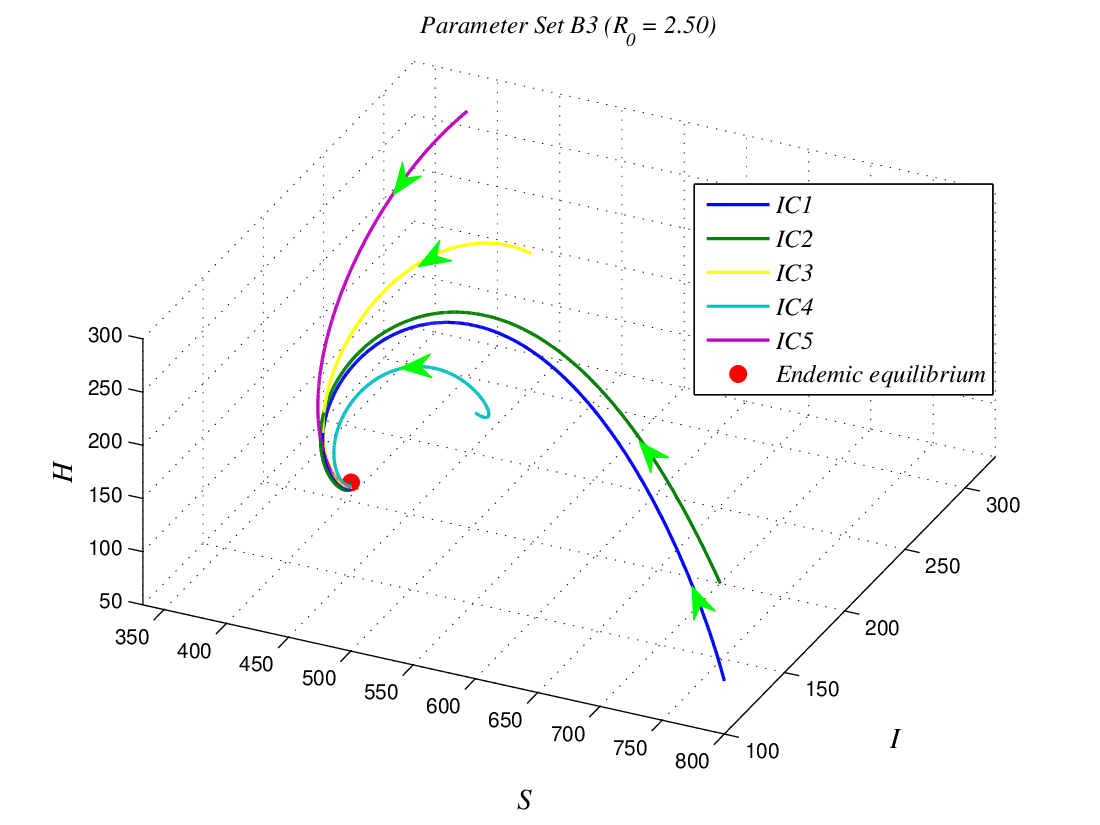}
\caption{Dynamics of \eqref{eq:1new} for the parameter set B3 in Table \ref{Table3}}\label{Fig:8}
\end{figure}
\begin{figure}[H]
\centering
\includegraphics[height=9.5cm,width=12cm]{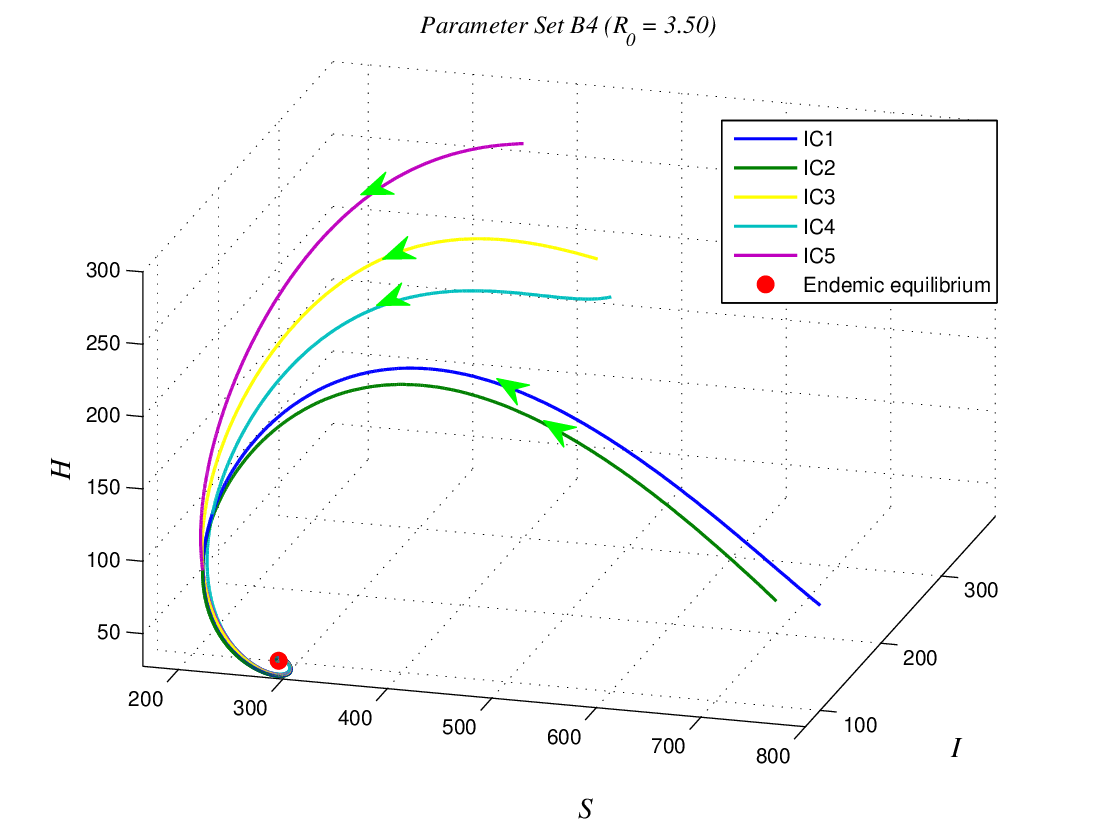}
\caption{Dynamics of \eqref{eq:1new} for the parameter set B4 in Table \ref{Table3}}\label{Fig:9}
\end{figure}
\begin{figure}[H]
\centering
\includegraphics[height=9.5cm,width=12cm]{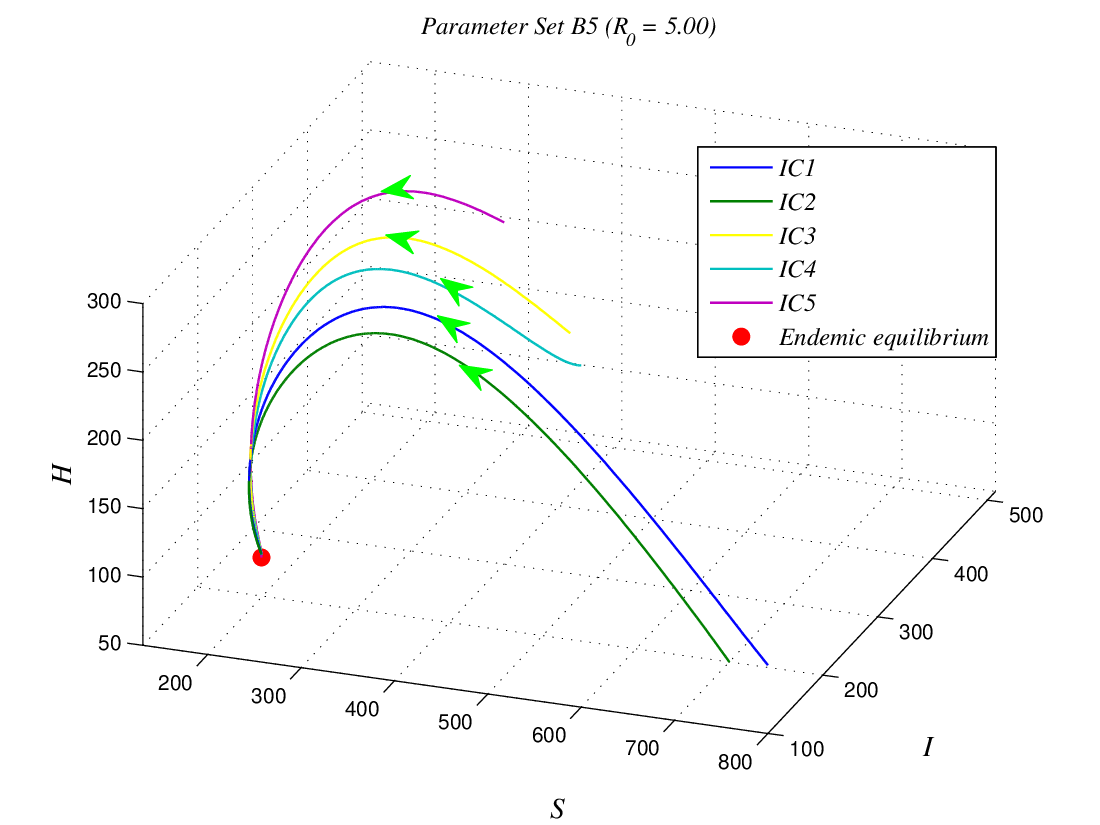}
\caption{Dynamics of \eqref{eq:1new} for the parameter set B5 in Table \ref{Table3}}\label{Fig:10}
\end{figure}
Before ending this subsection, we compute sensitivity indices  of the basic reproduction number $\mathcal{R}_0$ with respect to controllable parameters in \eqref{eq:1new}. We recall from \cite{Chitnis} that: The normalized forward sensitivity index of a variable, $u$, that depends differentiably on a parameter, $p$, is defined as:
\begin{equation*}
\Upsilon_p^u:=\frac{\partial u}{\partial p} \times \frac{p}{u}.
\end{equation*}%
The computed sensitivity indices using the parameter set B4 are shown in Figure \ref{Fig:11}. They indicate how strongly each parameter affects the disease transmission dynamics, which in turn helps formulate effective prevention and control strategies.
\begin{figure}[H]
\centering
\includegraphics[height=9.5cm,width=15cm]{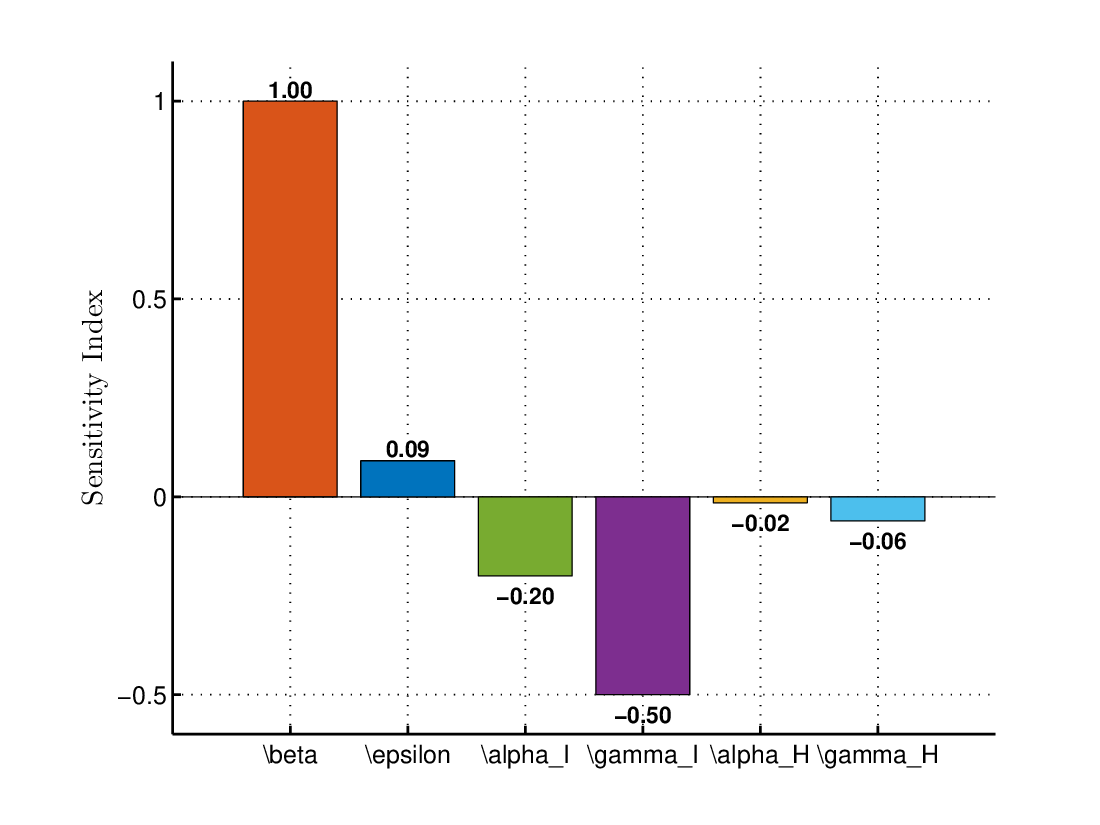}
\caption{The sensitivity indices evaluated by using the parameter set B4}\label{Fig:11}
\end{figure}
\subsection{An Example in Health Insurance Pricing}
In this section, we consider an example that demonstrates the application of the proposed model in insurance premium determination. First, we take the values of the benefit amounts in \eqref{eq:Pi} as follows
\begin{equation*}
b_I = 1.0, \quad b_H = 20, \quad d = 100.
\end{equation*}
Then, the premium rate $\pi$ defined in \eqref{eq:Pi1}, which represents the zero-profit condition, are computed as in Table \ref{Pi}.
\begin{table}[H]
\centering
\caption{The values of $\pi$ for each parameter set and initial condition in Tables \ref{Table3} and \ref{Table4}.}\label{Pi}
\begin{tabular}{cccccccccccccccc}
\hline
Parameter set &Initial data& $S_0$ & $I_0$ & $H_0$ & $\pi$ \\
\hline
B1 & IC1 & 800 & 100 & 100 & 1.93896 \\
B1 & IC2 & 700 & 200 & 50 & 1.91942 \\
B1 & IC3 & 500 & 250 & 250 & 2.26632 \\
B1 & IC4 & 600 & 100 & 300 & 2.24263 \\
B1 & IC5 & 400 & 300 & 300 & 2.36076 \\
B2 & IC1 & 800 & 100 & 100 & 3.70673 \\
B2 & IC2 & 700 & 200 & 50 & 3.67883 \\
B2 & IC3 & 500 & 250 & 250 & 4.04164 \\
B2 & IC4 & 600 & 100 & 300 & 4.01615 \\
B2 & IC5 & 400 & 300 & 300 & 4.13672 \\
B3 & IC1 & 800 & 100 & 100 & 6.33601 \\
B3 & IC2 & 700 & 200 & 50 & 6.38258 \\
B3 & IC3 & 500 & 250 & 250 & 6.85861 \\
B3 & IC4 & 600 & 100 & 300 & 6.73363 \\
B3 & IC5 & 400 & 300 & 300 & 6.99956 \\
B4 & IC1 & 800 & 100 & 100 & 5.83294 \\
B4 & IC2 & 700 & 200 & 50 & 5.7334 \\
B4 & IC3 & 500 & 250 & 250 & 6.4421 \\
B4 & IC4 & 600 & 100 & 300 & 6.43625 \\
B4 & IC5 & 400 & 300 & 300 & 6.62139 \\
B5 & IC1 & 800 & 100 & 100 & 16.3427 \\
B5 & IC2 & 700 & 200 & 50 & 16.3699 \\
B5 & IC3 & 500 & 250 & 250 & 17.1336 \\
B5 & IC4 & 600 & 100 & 300 & 16.9398 \\
B5 & IC5 & 400 & 300 & 300 & 17.3489 \\
\hline
\end{tabular}
\end{table}
For parameter set B2 and initial data IC1, we determine that the value of $\pi^*$, as defined in \ref{eq:Pi3}, is equal to $5.67475$. This value guarantees that the discounted value of future premiums remains sufficient to cover future liabilities. The graphs of the reserve level $V(t)$ in \eqref{Pi2} corresponding to $\pi = \pi^*$, $\pi = 90\% \pi^*$ and $\pi = 110\% \pi^*$ are depicted in Figures~\ref{Fig:12}-\ref{Fig:15}. These results support the theoretical analyses presented in Section \ref{Sec3}.
\begin{figure}[H]
\centering
\includegraphics[height=9.0cm,width=12cm]{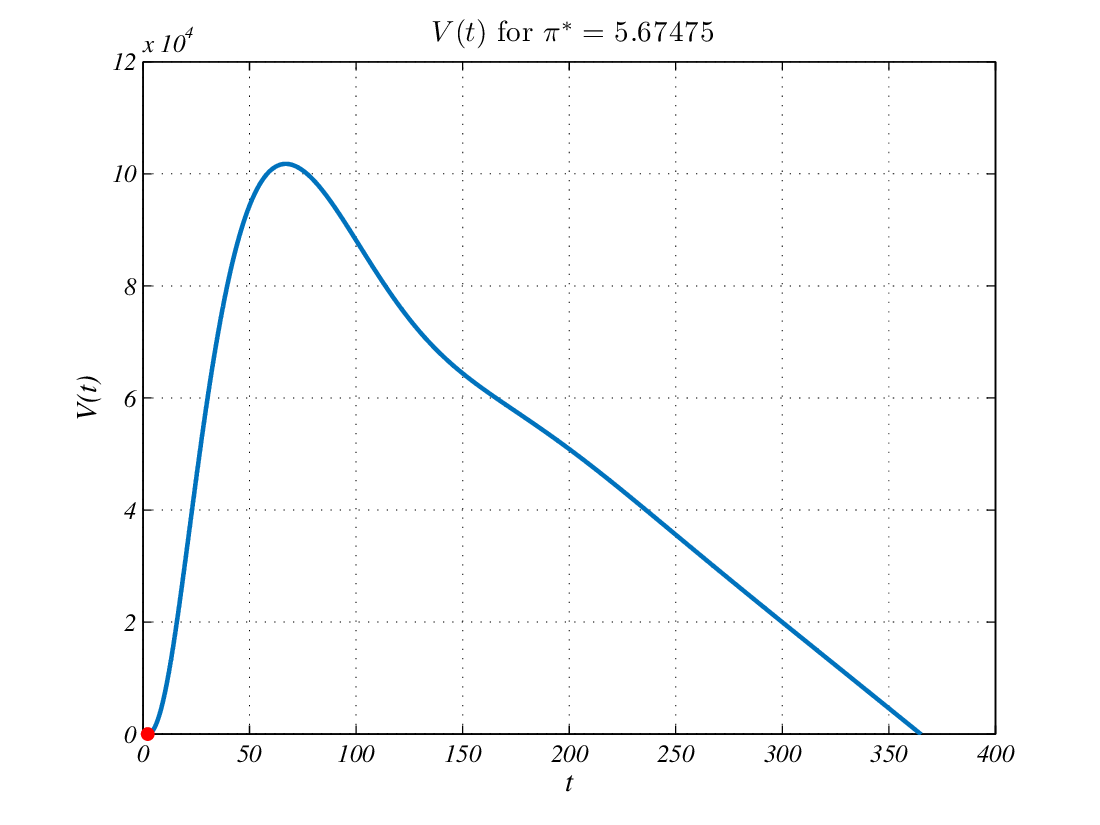}
\caption{The reserve level $V(t)$ with $\pi = \pi^*$}\label{Fig:12}
\end{figure}
\begin{figure}[H]
\centering
\includegraphics[height=9.0cm,width=12cm]{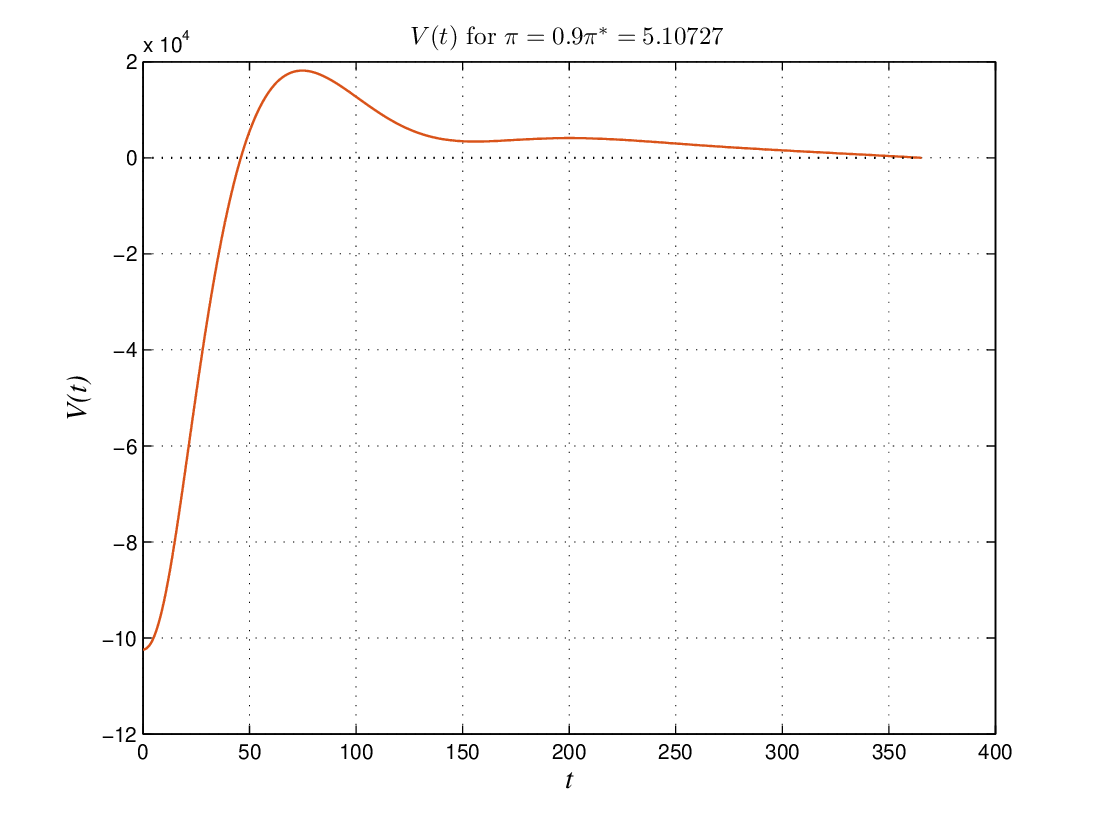}
\caption{The reserve level $V(t)$ with $\pi = 90\%\pi^*$}\label{Fig:13}
\end{figure}
\begin{figure}[H]
\centering
\includegraphics[height=9.0cm,width=12cm]{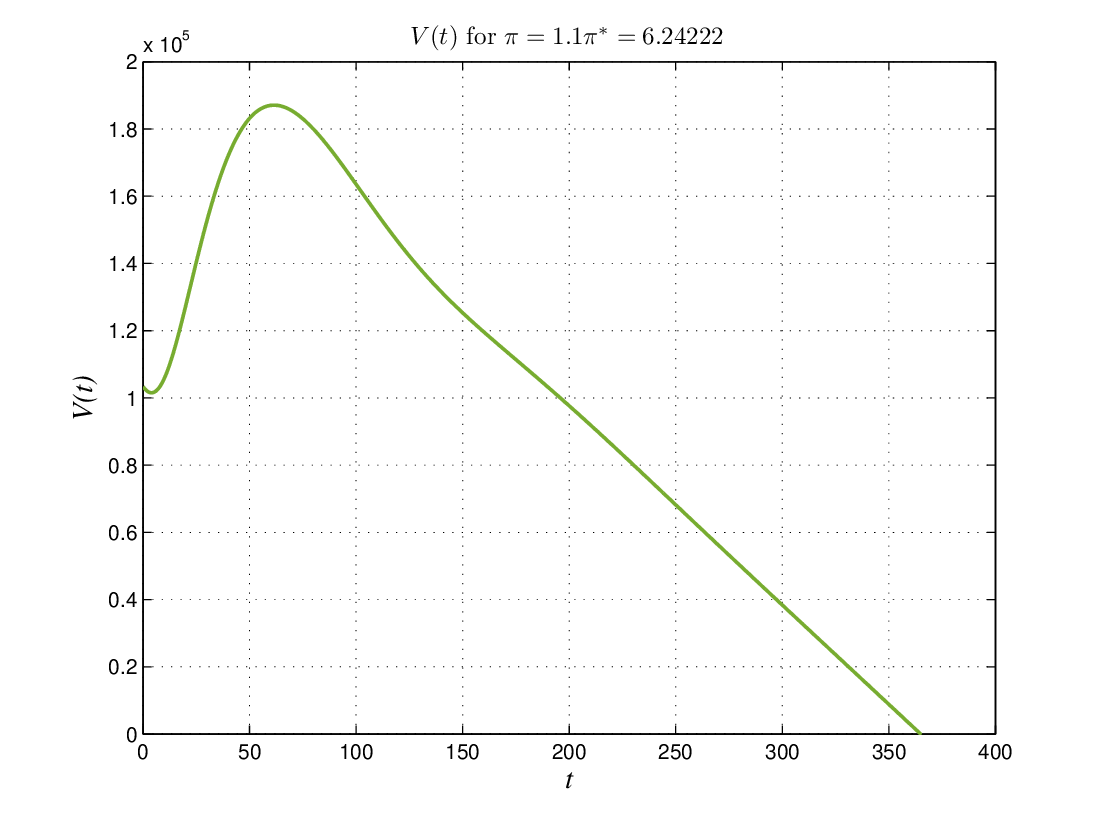}
\caption{The reserve level $V(t)$ with $\pi = 110\%\pi^*$}\label{Fig:14}
\end{figure}
\begin{figure}[H]
\centering
\includegraphics[height=9.0cm,width=12cm]{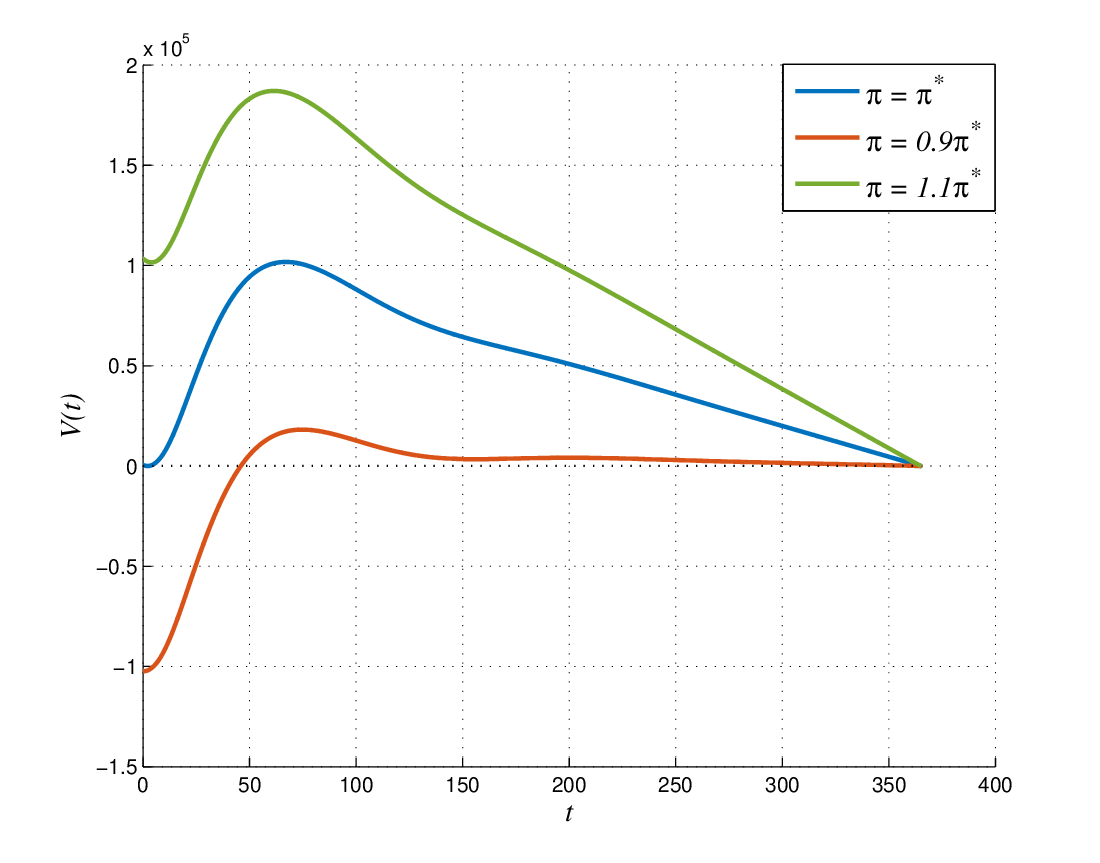}
\caption{the reserve level $V(t)$ in \eqref{Pi2} corresponding to $\pi = \pi^*$, $\pi = 90\% \pi^*$ and $\pi = 110\% \pi^*$.}\label{Fig:15}
\end{figure}
\section{Conclusions and Discussions}\label{Sec5}
In this study, we have investigated a modified version of the classical SIS model by incorporating hospitalization for treatment and disease-induced mortality, resulting in a so-called SISHD model. The dynamical analysis has demonstrated the positivity and boundedness of solutions, the basic reproduction number, as well as the existence and asymptotic stability of the DFE and DEE points. An application of the proposed model to health insurance pricing has been also presented, where the population dynamics derived from the SISHD model are used to estimate insurance costs. Numerical simulations have confirmed the theoretical results and illustrated the model’s potential applicability in the assessment and management of health insurance systems.

Future work may extend this framework by considering stochastic effects or age-structured populations to enhance its practical relevance. In particular, In particular, the following issues will be of special interest:
\begin{itemize}
\item the problem of parameter estimation from real-world data;
\item the consideration of factors affecting the disease transmission process, such as seasonal transmission and stochastic effects;
\item developing data-driven models to better capture real-world dynamics.
\end{itemize}
\section*{Acknowledgments}
\noindent
The co-first and last authors wish to thank the Vietnam Institute for Advanced Study in Mathematics (VIASM) for its financial support and the excellent working conditions. This work was completed while they were working at the VIASM.\\
We express our  gratitude to Dr. Oluwaseun F. Egbelowo (Merck \& Co., Inc.) for  the insightful discussions on the mathematical modeling and analysis of infectious diseases. We also thank Thai Tuan Dinh (VNU University of Science) for helpful discussions during the preparation of the manuscript. 
\section*{CRediT authorship contribution statement}
\noindent
\textbf{Tien Thinh Le} and \textbf{Tuan Chau Do}: Writing – review \& editing, Visualization, Validation, Software, Resources, Investigation, Formal analysis, Data curation.\\
\textbf{Nguyen Trong Hieu} and \textbf{Manh Tuan Hoang:} Writing – review \& editing, Writing – original draft, Visualization, Validation, Supervision, Software, Resources, Project administration, Methodology, Investigation, Formal analysis, Data curation, Conceptualization, Funding acquisition.\\
\textbf{Tien Thinh Le} and  \textbf{Tuan Chau Do} are contributed equally to this work
as co-first authors.
\section*{Availability of data and material}
The authors confirm that the data supporting the findings of this study are available within the article [and/or] its supplementary materials.
\section*{Competing interests}
We declare that there are no conflicts of interest with regards to the publication of this paper.

\bibliographystyle{amsalpha}

\begin{thebibliography}{999}
\bibitem{Allen}
L. J. S. Allen, An Introduction to Mathematical Biology, Prentice Hall, 2007.

\bibitem{Ascher}
U. M. Ascher, L. R. Petzold, Computer Methods for Ordinary Differential Equations and Differential-Algebraic Equations, Society for Industrial and Applied Mathematics, Philadelphia, 1998.

\bibitem{Brauer1}
F. Brauer, C. Castillo-Chavez, Mathematical Models in Population Biology and Epidemiology, Springer, Berlin, 2012.

\bibitem{Brauer2}
F. Brauer, Mathematical epidemiology: Past, present, and future, Infectious Disease Modelling 2(2017) 113-127.

\bibitem{Capasso}
V. Capasso, G. Serio, A generalization of the Kermack-McKendrick deterministic epidemic model, Mathematical Biosciences 42(1978) 43-61.

\bibitem{Castillo-Chavez}
C. Castillo-Chavez, Z. Feng, and W. Huang, On the computation of $R_0$ and its role on
global stability, in Mathematical Approaches for Emerging and Reemerging Infectious Diseases: Models, Methods and Theory, C. Castillo-Chavez, S. Blower, P. van den Driessche,
D. Kirschner, and A.-A. Yakubu, eds., Springer, Berlin, 2002, pp. 229-250.

\bibitem{Chernov}
A. A. Chernov, A. A. Shemendyuk, and M. Y. Kelbert, Fair insurance premium rate in connected SEIR
model under epidemic outbreak, Mathematical Modelling of Natural Phenomena 16(2021) 34.
\bibitem{Chitnis}
N. Chitnis, J. M. Hyman, Jim. M. Cushing, Determining Important Parameters in the Spread of Malaria
Through the Sensitivity Analysis of a Mathematical Model, Bulletin of Mathematical Biology (2008) 70: 1272-1296.
%
\bibitem{Feng}
R. Feng, J. Garrido, Actuarial applications of epidemiological models, North American Actuarial
Journal 15(2011) 112-136.

\bibitem{Feng1}
R. Feng, J. Garrido, L. Jin, S.-H. Loke, L. Zhang, Epidemic compartmental models and their insurance
applications,  In: Boado-Penas, M.d.C., Eisenberg, J., Sahin S. (eds) Pandemics: Insurance and Social Protection. Springer Actuarial. Springer, Cham. \url{https://doi.org/10.1007/978-3-030-78334-1_2}

\bibitem{Francis}
L. Francis, M. Steffensen, Individual life insurance during epidemics, Annals of Actuarial Science 18(2023) 152-175.

\bibitem{Gray}
A. Gray, D. Greenhalgh, L. Hu, X. Mao, J. Pan, A Stochastic Differential Equation SIS Epidemic Model, SIAM Journal on Applied Mathematics 7(2011) 876-902.

\bibitem{Kermack}
W. O. Kermack, A. G. McKendrick, A contribution to the mathematical theory of epidemics, Proceedings of the Royal Society of London - Series A 115 (1927) 700-721.

\bibitem{Lan}
G. Lan, Y. Huang, C. Wei, S. Zhang, A stochastic SIS epidemic model with saturating contact rate, Physica A 529 (2019) 121504.

\bibitem{Martcheva}
M. Martcheva, An Introduction to Mathematical Epidemiology, Springer, New York, 2015.

\bibitem{McNabb}
A. McNabb, Comparison theorems for differential equations, Journal of Mathematical Analysis and Applications 119(1986) 417-428.

\bibitem{Nkeki}
C. I. Nkeki, E. H. Iroh, On epidemiological and actuarial analyses of health insurance models for
communicable diseases, North American Actuarial Journal 29(2024) 422-451.

\bibitem{Nkeki1}
C. I. Nkeki, E. H. Iroh, Epidemiological and health insurance models for a communicable disease,
Annals of Financial Economics 19(2024) 2450007.

\bibitem{Smith}
H. L. Smith, P. Waltman, The Theory of the Chemostat: Dynamics of Microbial Competition, Cambridge University Press, 1995.

\bibitem{vandenDriessche}
P. van den Driessche, J. Watmough, Reproduction numbers and sub-threshold endemic
equilibria for compartmental models of disease transmission, Mathematical Biosciences
180 (2002) 29-48.

\bibitem{Zhai}
C. Zhai, P. Chen, Z. Jin, T. K. Siu, Epidemic modelling and actuarial applications for pandemic insurance: a case study of Victoria, Australia, Annals of Actuarial Science 18(2024) 242-269.

\bibitem{Zinihi}
A. Zinihi, M. Ehrhardt, M. R. S. Ammi, Actuarial Analysis of an Infectious Disease Insurance
based on an SEIARD Epidemiological Model, arXiv:2508.06580v1.

%
%
%
%
%
%
%
%
%
%
\end{thebibliography}

\end{document}